\documentclass[11pt]{amsart}

 \calclayout

\usepackage{amsmath,amssymb,color}
\usepackage[dvipdfm]{graphicx}

\numberwithin{equation}{section}
\newtheorem{thm}{Theorem}[section]

\newtheorem{lem}[thm]{Lemma}
\newtheorem{prop}[thm]{Proposition}

\newtheorem{cor}[thm]{Corollary}

\DeclareMathOperator{\Aut}{Aut}

\begin{document}

\title{Uniqueness of the direct decomposition of toric manifolds}
\author{Miho Hatanaka}
\address{Department of Mathematics, Osaka City University, Sumiyoshi-ku, Osaka 558-8585, Japan.}
\email{hatanaka.m.123@gmail.com}
\date{\today}
\maketitle

\begin{abstract}
In this paper, we study the uniqueness of the direct decomposition of a toric manifold.
We first observe that the direct decomposition of a toric manifold as \emph{algebraic varieties} is unique up to order of the factors.  An algebraically indecomposable toric manifold happens to decompose as smooth manifold and no criterion is known 
for two toric manifolds to be diffeomorphic, so the unique decomposition problem for toric manifolds as \emph{smooth manifolds} is highly nontrivial and nothing seems known for the problem so far.  
We prove that this problem is affirmative if the complex dimension of each factor in the decomposition is less than or equal to two. 
A similar argument shows that the direct decomposition of a smooth manifold into copies of $\mathbb{C}P^1$ and simply connected closed smooth 4-manifolds with smooth actions of $(S^1)^2$ is unique up to order of the factors.
\end{abstract}

\section{Introduction}

A \emph{toric variety} is a normal algebraic variety of complex dimension $n$ with a complex torus action having an open dense orbit. 
The family of toric varieties one-to-one corresponds to that of fans which are objects in combinatorics. 
Via this correspondence, we can describe geometrical properties of toric varieties in terms of the corresponding fans. 
A toric variety may not be compact and nonsingular, however, this paper deals with compact nonsingular toric varieties, called \emph{toric manifolds}. 

We say that a toric manifold is \emph{algebraically indecomposable} if it does not decompose into the product of two toric manifolds of positive dimension \emph{as varieties}.  Using the bijective correspondence between toric varieties and fans, one can see that the direct decomposition of a toric manifold into algebraically indecomposable toric manifolds as algebraic varieties is unique up to order of the factors (Theorem~\ref{theo:2.1}). 

If two toric manifolds are isomorphic as varieties, then they are diffeomorphic, but the converse is not true in general and no criterion is known for two toric manifolds to be diffeomorphic.  One intriguing problem in this direction is the following problem posed in \cite{ma-su08}.

\medskip
\noindent
{\bf Cohomological rigidity problem for toric manifolds} (\cite{ma-su08}).  Are two toric manifolds diffeomorphic (or homeomorphic) if their cohomology rings with integer coefficients are isomorphic as graded rings?

\medskip
\noindent
No counterexample and some partial affirmative solutions are known to the problem above, see \cite{ch-ma-su11} for the recent development.   

An algebraically indecomposable toric manifold happens to decompose into the product of two toric manifolds of positive dimension as \emph{smooth manifolds}.  Hirzebruch surfaces except $\mathbb{C}P^1\times\mathbb{C}P^1$ with vanishing second Stiefel-Whitney classes are such examples.  We say that a toric manifold is \emph{differentially indecomposable} if it does not decompose into the product of two toric manifolds of positive dimension \emph{as smooth manifolds}.  

\medskip
\noindent
{\bf Unique decomposition problem for toric manifolds} (\cite{masu12}).  Is the direct decomposition of a toric manifold into the product of differentially indecomposable toric manifolds unique up to order of the factors?

\medskip 

It has recently been shown in \cite{ch-ma-ou10} that the unique decomposition property holds for real Bott manifolds which are a special class of \emph{real} toric manifolds.  Real Bott manifolds are compact flat manifolds and it is shown in \cite{char65} that there are non-diffeomorphic compact flat manifolds whose products with $S^1$ are diffeomorphic.  This means that the unique decomposition property does not hold for general compact flat manifolds while it does for the special class of compact flat manifolds consisting of real Bott manifolds.  

As far as the author knows, nothing is known for the unique decomposition problem for toric manifolds.  
In this paper, we show that it is affirmative if the complex dimension of every factor in the product is less than or equal to two (Theorem~\ref{theo:3.1}).  We also prove that the cohomological rigidity problem is affirmative for those products.  
Note that a toric manifold of complex dimension one is diffeomorphic to $\mathbb{C} P^1$ and that of complex dimension two is diffeomorphic to $\mathbb{C} P^1\times \mathbb{C} P^1$ or $\mathbb{C}P^2\sharp q \overline{\mathbb{C}P^2}$ $(q\ge 0)$.  

Simply connected closed smooth 4-manifolds with smooth actions of $(S^1)^2$ are of the form  
\begin{equation} \label{eq:1.1}
S^4 \sharp p\mathbb{C}P^2 \sharp q\overline{\mathbb{C}P^2} \sharp r(\mathbb{C}P^1 \times \mathbb{C}P^1)\quad (p+q+r\ge 0)
\end{equation}
(see \cite{or-ra70}).  These manifolds are not diffeomorphic to the product of two manifolds of positive dimension unless $p=q=0$ and $r=1$.  Our method used to prove Theorem~\ref{theo:3.1} can be applied to products of copies of $\mathbb{C}P^1$ and manifolds in \eqref{eq:1.1} and yields a more general result (Theorem~\ref{theo:4.1}) than Theorem~\ref{theo:3.1}.     

This paper is organized as follows.  In Section 2, we prove the uniqueness of the direct decomposition of a toric manifold into algebraically indecomposable toric manifolds as algebraic varieties. 
The key fact used to prove it is that two toric manifolds are isomorphic as algebraic varieties if and only if the corresponding two fans are isomorphic.  
Unlike this, a useful criterion for two toric manifolds to be diffeomorphic is not known. 
In Section 3, we prove that the direct decomposition of a toric manifold into differentially indecomposable toric manifolds is unique up to order of the factors if the complex dimension of each factor is less than or equal to two.  In Section 4, we apply the idea developed in Section 3 to products of copies of $\mathbb{C}P^1$ and manifolds in \eqref{eq:1.1}.   

\section{Direct decomposition of toric manifolds as algebraic varieties} 

We briefly review toric geometry and refer the reader to \cite{fult93} and \cite{oda88} for details.   
A \textit{toric variety} is a normal algebraic variety of complex dimension $n$ with an algebraic action of a complex torus $(\mathbb{C}^*)^n$ having an open dense orbit. The fundamental theorem in toric geometry says that the category of toric varieties of (complex) dimension $n$ is isomorphic to the category of fans of (real) dimension $n$. Here, a \textit{fan} $\Delta$ of dimension $n$ is a collection of rational strongly convex polyhedral cones in $\mathbb{R}^n$ satisfying the following conditions: \ \\
EEach face of a cone in $\Delta$ is also a cone in $\Delta$.\ \\
EThe intersection of two cones in $\Delta$ is a face of each.\ \\
A rational strongly convex polyhedral cone in $\mathbb{R}^n$ is a cone with apex at the origin, generated by a finite number of vectors;
grationalh means that it is generated by vectors in the lattice $\mathbb{Z}^n$, and gstrongh convexity that it contains no line through the origin.
The union of cones in the fan $\Delta$ coincides with $\mathbb{R}^n$ if and only if the corresponding toric variety is compact, and the generators of each cone in $\Delta$ are a part of a basis of $\mathbb{Z}^n$ if and only if the corresponding toric variety is nonsingular. In this paper, we will treat only compact nonsingular toric varieties and call them \textit{toric manifolds}.

The fundamental theorem in toric geometry implies that two toric manifolds $M$ and $N$ of complex dimension $n$ are weakly equivariantly isomorphic as algebraic varieties if and only if the corresponding fans are isomorphic, i.e., there is an automorphism of $\mathbb{Z}^n$ sending cones to cones in the corresponding fans. Here a map $f\colon M\to N$ is said to be weakly equivariant if there is an automorphism $\rho$ of $(\mathbb{C}^*)^n$ such that $f(gx)=\rho(g)f(x)$ for any $g\in (\mathbb{C}^*)^n$ and $x\in M$.  

\begin{prop} \label{prop:2.1}
Two toric manifolds are isomorphic as algebraic varieties if and only if they are weakly equivariantly isomorphic as algebraic varieties.  Therefore, two toric manifolds are isomorphic as algebraic varieties if and only if their corresponding fans are isomorphic. 
\end{prop}

\begin{proof} 
This proposition is well-known but since there seems no literature, we shall sketch the proof.  

It suffices to prove the \lq\lq only if" part in the former statement because the \lq\lq if" part is trivial and the latter statement follows from the former statement and the fundamental theorem in toric geometry as remarked above. 
 Let $\Aut(M)$ be the group of automorphisms of a toric manifold $M$. This is a (finite dimensional) algebraic group, and the torus $T_M=(\mathbb{C}^*)^n$ acting on $M$ is a subgroup of $\Aut(M)$, in fact, it is a maximal torus in $\Aut(M)$.  Now, let $f$ be an isomorphism (as algebraic varieties) from $M$ to another toric manifold $N$.  Then $f$ induces a group isomorphism ${\hat f}\colon \Aut(N)\to \Aut(M)$ mapping $g\in \Aut(N)$ to $f^{-1}\circ g\circ f \in \Aut(M)$.  Since ${\hat f}(T_N)$ is a maximal torus in $\Aut(M)$ and all maximal tori in an algebraic group  are conjugate to each other, there exists $h\in\Aut(M)$ satisfying  ${\hat f}(T_N)=hT_Mh^{-1}$.
 Then $f\circ h$ is a weakly equivariant isomorphism from $M$ to $N$.
\end{proof}

We say that a toric manifold is \textit{algebraically indecomposable} if it does not decompose into the product of two toric manifolds of positive dimension as algebraic varieties.  Again, the fundamental theorem in toric geometry implies that a toric manifold is algebraically indecomposable if and only if the corresponding fan is \textit{indecomposable}, i.e., it does not decompose into the product of two fans of positive dimension. 


\begin{thm} \label{theo:2.1}
 The direct decomposition of a toric manifold into algebraically indecomposable toric manifolds as algebraic varieties is unique up to order of the factors.  Namely, if $M_i$ $(1\le i\le k)$ and $M_j^\prime$ $(1\le j\le \ell)$ are algebraically indecomposable toric manifolds and $\prod_{i=1}^kM_i$ and $\prod_{j=1}^\ell M_j^\prime$ are isomorphic as algebraic varieties, then $k=\ell$ and there exists an element $\sigma$ in the symmetric group $S_k$ on $k$ letters such that $M_i$ is isomorphic to $M_{\sigma(i)}^\prime$ as algebraic varieties for all $1\le i\le k$.
 \end{thm}

\begin{proof}
Denote the fan of $M_i$ by $\Delta_i$ and that of $M_j^\prime$ by $\Delta_j^\prime$, and let $\psi$ be an isomorphism from $\prod_{i=1}^k \Delta_i$ to $\prod_{j=1}^\ell \Delta_j^\prime$.  Let $p_j$ be the projection from $\prod_{j=1}^\ell \Delta_j^\prime$ onto $\Delta_j^\prime$.  Since an edge in $\Delta_i$ maps to an edge in $\prod_{j=1}^\ell \Delta_j^\prime$ by $\psi$, the image $\psi(\Delta_i)$ coincides with the product $\prod_{j=1}^\ell p_{j}(\psi(\Delta_i))$.  This together with the indecomposability of $\Delta_i$ implies that $p_j(\psi(\Delta_i))$ consists of only the origin except for one $j$, namely $\psi(\Delta_i)$ is contained in some $\Delta_j^\prime$.  Applying the same argument to $\psi^{-1}$, one concludes that $\psi(\Delta_i)=\Delta_j^\prime$. This together with Proposition~\ref{prop:2.1} proves the theorem.  
\end{proof}

The following corollary follows from Theorem~\ref{theo:2.1}.

\begin{cor}[cancellation]. 
 Let $M,M^\prime$ and $M^{\prime \prime}$ be toric manifolds.
 If the direct products $M \times M^{\prime \prime}$ and $M^{\prime} \times M^{\prime \prime}$ are isomorphic as varieties, then so are $M$ and $M^{\prime}$. 
\end{cor}

\section{Direct decomposition of toric manifolds as smooth manifolds}

In this section, we will consider the direct decomposition of toric manifolds as smooth manifolds.
We say that a toric manifold $M$ is \textit{differentially indecomposable} if $M$ does not decompose into two toric manifolds of positive dimension as smooth manifolds. 
We note that the algebraic indecomposability does not imply the differential indecomposability for toric manifolds. 
For example, the Hirzebruch surface $F_a$ $(a\in\mathbb{Z})$ corresponding to the fan described below is algebraically indecomposable unless $a=0$ but diffeomorphic to $\mathbb{C}P^1 \times \mathbb{C}P^1$ as smooth manifolds if $a$ is even.
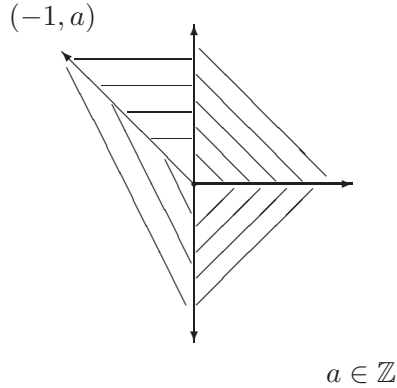
\begin{center}
 \begin{picture}(160,160)
  \put(80,80){\circle*{2}}
  \linethickness{0.5pt} \put(80,80){\vector(1,0){60}}
  \linethickness{0.5pt} \put(80,80){\vector(0,1){60}}
  \linethickness{1.5pt} \put(80,80){\vector(-1,1){50}}
  \linethickness{0.5pt} \put(80,80){\vector(0,-1){60}}
  \linethickness{0.2pt} \put(81,91){\line(1,-1){10}}
  \linethickness{0.2pt} \put(81,101){\line(1,-1){20}}
  \linethickness{0.2pt} \put(81,111){\line(1,-1){30}}
  \linethickness{0.2pt} \put(81,121){\line(1,-1){40}}
  \linethickness{0.2pt} \put(82,131){\line(1,-1){48}}
  \linethickness{0.2pt} \put(79,97){\line(-1,0){15}}
  \linethickness{0.2pt} \put(79,107){\line(-1,0){24}}
  \linethickness{0.2pt} \put(79,117){\line(-1,0){34}}
  \linethickness{0.2pt} \put(79,127){\line(-1,0){44}}
  \linethickness{0.2pt} \put(69,89){\line(1,-2){10}}
  \linethickness{0.2pt} \put(49,110){\line(1,-2){30}}
  \linethickness{0.2pt} \put(32,124){\line(1,-2){45}}
  \linethickness{0.2pt} \put(81,64){\line(1,1){14}}
  \linethickness{0.2pt} \put(81,54){\line(1,1){24}}
  \linethickness{0.2pt} \put(81,44){\line(1,1){34}}
  \linethickness{0.2pt} \put(81,34){\line(1,1){44}}
  \put(10,140){$(-1,a)$} \put(130,5){$a \in \mathbb{Z}$}
 \end{picture}
\end{center}

Toric manifolds of complex dimension one are diffeomorphic to $\mathbb{C}P^1$, and those of complex dimension two are diffeomorphic to $\mathbb{C}P^1 \times \mathbb{C}P^1$ or $\mathbb{C}P^2 \sharp q\overline{\mathbb{C}P^2}$ $(q \in \mathbb{Z}_{\geq 0})$.
The purpose of this section is to prove the following theorem. 

\begin{thm} \label{theo:3.1}
Let $M_i$ $(1\le i\le k)$ and $M_j'$ $(1\le j\le \ell)$ be differentially indecomposable toric manifolds of complex dimension less than or equal to two. 
If $H^*(\prod_{i=1}^kM_i;\mathbb{Z})$ and $H^*(\prod_{j=1}^\ell M_j^\prime;\mathbb{Z})$ are isomorphic as graded rings, then $k=\ell$ and there exists an element $\sigma$ in the symmetric group $S_k$ on $k$ letters such that $M_i$ and $M_{\sigma(i)}^\prime$ are diffeomorphic for all $1\le i\le k$.
Therefore, the cohomological rigidity problem and the unique decomposition problem mentioned in the Introduction are both affirmative for products of differentially indecomposable toric manifolds of complex dimension less than or equal to two.
\end{thm}

For the proof of this theorem, we consider 
\begin{equation} \label{eq:A}
A(X;R) = \{u \in H^2(X;R)\backslash \{0\} \mid u^2=0 \} 
\end{equation}
for a topological space $X$ and a commutative ring $R$. 

\begin{lem} \label{3.2}
Let $R$ be $\mathbb{Z}$ or a field, and let $X_i$ $(1\le i\le k)$ be a topological space such that $H^q(X_i;R)$ is finitely generated for any $q$ and $H^{1}(X_i;R)=H^3(X_i;R)=0$. 
Moreover, when $R=\mathbb{Z}$, we suppose that $H^q(X_i;\mathbb{Z})$ $(q\le 4)$ is a free module.
iToric manifolds satisfy these conditions.j
Then $$A(\prod_{i=1}^kX_i ; R) = \bigsqcup_{i=1}^k A(X_i;R).$$ 
\end{lem}

\begin{proof}
By K\"{u}nneth formula, $H^2(\prod_{i=1}^kX_i;R)$ is isomorphic to $\bigoplus_{i=1}^k H^2(X_i;R)$.
 So an element $u$ in $H^2(\prod_{i=1}^k X_i;R)$ can be written as $u = u_1 + \dots + u_k$ $(u_i\in H^2(X_i;R))$.
Again, by K\"{u}nneth formula,
 \[ H^4(\prod_{i=1}^k X_i;R)\cong \Big(\bigoplus_{i=1}^kH^4(X_i;R)\Big)\oplus\Big(\bigoplus_{1\le i<j\le k}H^2(X_i;R)\otimes H^2(X_j;R)\Big) \]
 and via this isomorphism
 \[ u^2=\sum_{i=1}^k u_i^2+2\sum_{1\le i<j\le k}u_i\otimes u_j. \]
 So if $u^2=0$, then $u_i=0$ except one $i$.
 Therefore, the lemma holds.
\end{proof}

Differentially indecomposable toric manifolds of complex dimension less than or equal to two are diffeomorphic to $\mathbb{C}P^1$ or $\mathbb{C}P^2 \sharp q\overline{\mathbb{C}P^2}$\ $(q \in \mathbb{Z}_{\geq 0})$.
Their cohomology rings are as follows:
\begin{equation} \label{eq:3.1}
\begin{split}
&H^*(\mathbb{C}P^1 ; R) \cong R[x]/(x^2=0) \\
&H^*(\mathbb{C}P^2 \sharp q\overline{\mathbb{C}P^2} ; R) \cong R[x,y_1,\dots,y_q]/\bigl(x^2 = -y_i^2,\  xy_i=0\ (\forall i),\ y_iy_j=0\ (i \neq j)\bigr) 
\end{split}
\end{equation}

\begin{lem} \label{3.3}
\begin{enumerate}
\item[(1)] $A(\mathbb{C}P^1 ; R) \cong \{ a \in R \backslash \{0\} \}.$
In particular, $A(\mathbb{C}P^1;\mathbb{R})$ consists of two one dimensional connected components, and $A(\mathbb{C}P^1;\mathbb{Z}/2)$ consists of one element. 
\item[(2)] $A(\mathbb{C}P^2 \sharp q\overline{\mathbb{C}P^2} ; R) \cong \{ (a,b_1, \dots ,b_q) \in R^{q+1}\backslash \{0\} \mid a^2 = b_1^2 + \dots + b_q^2 \}.$
In particular, $A(\mathbb{C}P^2;\mathbb{R})$ and $A(\mathbb{C}P^2;\mathbb{Z}/2)$ are empty, $A(\mathbb{C}P^2 \sharp \overline{\mathbb{C}P^2} ; \mathbb{R})$ consists of four one dimensional connected components, and $A(\mathbb{C}P^2 \sharp \overline{\mathbb{C}P^2} ; \mathbb{Z}/2)$ consists of one element.
When $q\ge 2$, $A(\mathbb{C}P^2 \sharp q\overline{\mathbb{C}P^2} ; \mathbb{R})$ consists of two $q$ dimensional connected components. 
\end{enumerate}
\end{lem}

\begin{proof}
(1) This easily follows from the former isomorphism in \eqref{eq:3.1}.

(2) Using the latter isomorphism in \eqref{eq:3.1}, one can write an element $u$ in $H^2(\mathbb{C}P^2 \sharp q\overline{\mathbb{C}P^2} ; R)$ as 
\[ u = ax + b_1y_1 + \dots + b_qy_q\ \ (a, b_1, \dots, b_q \in R), \]
so we have 
$u^2 = (a^2 - b_1^2 - \dots -b_q^2)x^2$, which implies (2).  
\end{proof}

\begin{proof}[Proof of Theorem~\ref{theo:3.1}] 
Let $m$ (resp, $m_q$) be the number of $M_i$'s diffeomorphic to $\mathbb{C}P^1$ (resp, $\mathbb{C}P^2\sharp q\overline{\mathbb{C}P^2}$). 
Similarly, let $m'$ (resp, $m'_q$) be the number of $M'_j$'s diffeomorphic to $\mathbb{C}P^1$ (resp, $\mathbb{C}P^2\sharp q\overline{\mathbb{C}P^2}$). 
Then
\begin{equation} \label{eq:3.1-1}
\begin{split}
M&:=\prod_{i=1}^kM_i=(\mathbb{C}P^1)^m\times \prod_{q\ge 0} (\mathbb{C}P^2\sharp q\overline{\mathbb{C}P^2})^{m_q}\\
M^\prime &:=\prod_{j=1}^\ell M_j^\prime=(\mathbb{C}P^1)^{m'}\times \prod_{q\ge 0} (\mathbb{C}P^2\sharp q\overline{\mathbb{C}P^2})^{m_q'}.
\end{split}
\end{equation}

By assumption, $H^*(M;\mathbb{Z})$ and $H^*(M^\prime;\mathbb{Z})$ are isomorphic as graded rings, and an isomorphism between them induces an isomorphism between $H^*(M;R)$ and $H^*(M^\prime;R)$ for any commutative ring $R$ and a bijection between $A(M;R)$ and $A(M^\prime;R)$.
When $R=\mathbb{R}$, we compare the number of connected components of dimension $t$ in $A(M;\mathbb{R})$ and $A(M^\prime;\mathbb{R})$.
Since the bijection between $A(M;\mathbb{R})$ and $A(M^\prime;\mathbb{R})$ is a homeomorphism, we obtain 
\begin{equation} \label{eq:3.2}
2m+4m_1=2m'+4m_1',\qquad
2m_t=2m_t' \quad (t\ge 2)
\end{equation}
from Lemmas~\ref{3.2} and \ref{3.3}.
Moreover, comparing the number of elements in $A(M;\mathbb{Z}/2)$ and $A(M^\prime;\mathbb{Z}/2)$, we obtain 
\begin{equation} \label{eq:3.3}
m+m_1=m'+m_1'
\end{equation}
from the fact $m_t=m_t'$ $(t\ge 2)$ in \eqref{eq:3.2}, Lemmas~\ref{3.2} and \ref{3.3}.
The identities \eqref{eq:3.2} and \eqref{eq:3.3} imply $m=m'$ and $m_t=m_t'$ $(t\ge 1)$.
These together with the equality of the dimensions of $M$ and $M^\prime$ (which are respectively $m+2\sum_{t\ge 0} m_t$ and $m'+2\sum_{t\ge 0} m_t'$ by \eqref{eq:3.1-1}) imply $m_0=m_0'$.
Therefore the theorem is proved. 
\end{proof}


The following corollary follows from Theorem~\ref{theo:3.1}.

\begin{cor}[cancellation] 
 Let $M$, $M^\prime$ and $M^{\prime \prime}$ be products of toric manifolds of complex dimension less than or equal to two.  If $M \times M^{\prime \prime}$ and $M^\prime \times M^{\prime \prime}$ are diffeomorphic, then so are $M$ and $M^\prime$. 
\end{cor}

\section{Simply connected compact 4-manifolds with $(S^1)^2$-actions}

In this section, we show that the idea developed to prove Theorem~\ref{theo:3.1} works for products of $\mathbb{C}P^1$ and simply connected compact smooth 4-manifolds with smooth actions of compact torus $(S^1)^2$.
By Orlik-Raymond (\cite{or-ra70}), these 4-manifolds are diffeomorphic to 
\begin{equation} \label{eq:4.0}
S^4 \sharp p\mathbb{C}P^2 \sharp q\overline{\mathbb{C}P^2} \sharp r(\mathbb{C}P^1 \times \mathbb{C}P^1)\quad (p+q+r\ge 0).
\end{equation}

\begin{prop} \label{prop:4.1}
A manifold in \eqref{eq:4.0} is diffeomorphic to one of the following:
\[ S^4, \quad p\mathbb{C}P^2 \sharp q\overline{\mathbb{C}P^2}\ (p\ge q\ge 0,\ p+q\ge 1), \quad r(\mathbb{C}P^1 \times \mathbb{C}P^1)\ (r\ge 1). \]
Moreover these manifolds are not diffeomorphic to each other.
\end{prop}

\begin{proof}
This proposition must be known but since there seems no literature, we shall give a proof.  

\medskip
\noindent{\bf Claim} 
${\mathbb{C}P^2} \sharp (\mathbb{C}P^1 \times \mathbb{C}P^1)$ and $\overline{\mathbb{C}P^2} \sharp (\mathbb{C}P^1 \times \mathbb{C}P^1)$ are diffeomorphic to $\mathbb{C}P^2 \sharp 2\overline{\mathbb{C}P^2}$.

\medskip

The fan corresponding to the blow-up of $\mathbb{C}P^1 \times \mathbb{C}P^1$ and that of $\mathbb{C}P^2 \sharp \overline{\mathbb{C}P^2}$ are isomorphic, so $\overline{\mathbb{C}P^2} \sharp (\mathbb{C}P^1 \times \mathbb{C}P^1)$ and $\mathbb{C}P^2 \sharp 2\overline{\mathbb{C}P^2}$ are isomorphic as algebraic varieties, in particular, $\overline{\mathbb{C}P^2} \sharp (\mathbb{C}P^1 \times \mathbb{C}P^1)$ is diffeomorphic to $\mathbb{C}P^2 \sharp 2\overline{\mathbb{C}P^2}$.

Moreover ${\mathbb{C}P^2} \sharp (\mathbb{C}P^1 \times \mathbb{C}P^1)$ and $\overline{\mathbb{C}P^2} \sharp (\overline{\mathbb{C}P^1 \times \mathbb{C}P^1})$ are diffeomorphic, and since there is an orientation preserving diffeomorphism from $\overline{\mathbb{C}P^1 \times \mathbb{C}P^1}$ to $\mathbb{C}P^1 \times \mathbb{C}P^1$ (i.e., an orientation reversing diffeomorphism from $\mathbb{C}P^1 \times \mathbb{C}P^1$ to itself), $\overline{\mathbb{C}P^2} \sharp (\overline{\mathbb{C}P^1 \times \mathbb{C}P^1})$ is diffeomorphic to $\overline{\mathbb{C}P^2} \sharp ({\mathbb{C}P^1 \times \mathbb{C}P^1})$. 
So ${\mathbb{C}P^2} \sharp (\mathbb{C}P^1 \times \mathbb{C}P^1)$ and $\overline{\mathbb{C}P^2} \sharp (\mathbb{C}P^1 \times \mathbb{C}P^1)$ are diffeomorphic. 
Therefore the claim is proved.

\medskip

From the Claim above and the fact that $p\mathbb{C}P^2 \sharp q\overline{\mathbb{C}P^2}$ and $q\mathbb{C}P^2 \sharp p\overline{\mathbb{C}P^2}$ are diffeomorphic, we see that a manifold in \eqref{eq:4.0} is diffeomorphic to one of the manifolds in Proposition~\ref{prop:4.1}. 

We shall prove that the manifolds in Proposition~\ref{prop:4.1} are not diffeomorphic to each other.
The manifolds $p\mathbb{C}P^2 \sharp q\overline{\mathbb{C}P^2}$ are not spin manifolds (i.e., their second Stiefel-Whitney classes do not vanish) while $r(\mathbb{C}P^1 \times \mathbb{C}P^1)$ are spin manifolds. 
Therefore, they are not homotopy equivalent, in particular, not diffeomorphic. 
Euler characteristic $\chi$ and the absolute value of signature $\sigma$ are homotopy invariants, and
\begin{align*}
&\chi(p\mathbb{C}P^2 \sharp q\overline{\mathbb{C}P^2})=p+q+2,& &\sigma(p\mathbb{C}P^2 \sharp q\overline{\mathbb{C}P^2})=p-q\\
&\chi(r(\mathbb{C}P^1 \times \mathbb{C}P^1))=2r+2,& &\sigma(r(\mathbb{C}P^1 \times \mathbb{C}P^1))=0\\
&\chi(S^4)=2
\end{align*}
so the manifolds in Proposition~\ref{prop:4.1} are not homotopy equivalent to each other, in particular, they are not diffeomorphic to each other. 
\end{proof}


We find $A(M;R)$ in \eqref{eq:A} for the manifolds $M$ in Proposition~\ref{prop:4.1} and any commutative ring $R$.
Since 
\begin{equation*} 
\begin{split}
H^*(&p\mathbb{C}P^2 \sharp q\overline{\mathbb{C}P^2};R) \\
&\cong R[x_1,\dots,x_p,y_1,\dots,y_q]/(x_i^2=-y_j^2,\ x_iy_j=0 (\forall i,j),\ x_ix_j=0,\ y_iy_j=0 (\forall i \neq j)),\\
H^*(&r(\mathbb{C}P^1 \times \mathbb{C}P^1);R) \\
&\cong R[z_1,\dots,z_r,w_1,\dots,w_r]/(z_iw_i=z_jw_j,\ z_iz_j=w_iw_j=0 (\forall i,j),\ z_iw_j=0 (\forall i \neq j)),\\
H^*(&S^4;R) \cong R[x]/(x^2=0), 
\end{split}
\end{equation*}
we see that 
\begin{align}
A(&p\mathbb{C}P^2 \sharp q\overline{\mathbb{C}P^2} ; R)\notag\\
 &\cong \{(a_1, \dots ,a_p,b_1, \dots ,b_q) \in R^{p+q}\backslash \{0\} \mid a_1^2 + \dots + a_p^2 = b_1^2 + \dots + b_q^2 \},\label{eq:4.1-1}\\
A(&r(\mathbb{C}P^1 \times \mathbb{C}P^1) ; R) \notag\\
&\cong \{(c_1, \dots ,c_r,d_1, \dots ,d_r) \in R^{2r}\backslash \{0\} \mid c_1d_1 + \dots + c_rd_r = 0 \}, \label{eq:4.1-2}\\
A(&S^4;R) = \emptyset.\notag
\end{align}

\begin{lem} \label{4.2}
\begin{enumerate}
\item[(1)] $A(p\mathbb{C}P^2; \mathbb{R})$ is empty.
\item[(2)] When $p\ge q\ge 1$, $A(p\mathbb{C}P^2 \sharp q\overline{\mathbb{C}P^2};\mathbb{R})$ is homeomorphic to $S^{p-1}\times S^{q-1}\times\mathbb{R}$. 
\item[(3)] $A(r(\mathbb{C}P^1 \times \mathbb{C}P^1);\mathbb{R})$ is homeomorphic to $S^{r-1}\times S^{r-1}\times \mathbb{R}$.
\end{enumerate}
\end{lem}

\begin{proof} 
(1) This easily follows from \eqref{eq:4.1-1}.
 
 (2)  For each positive real number $c$, the set 
 \[ \{ (a_1, \dots, a_p, b_1, \dots, b_q) \in \mathbb{R}^{p+q} \backslash \{0\} \mid a_1^2 + \dots +a_p^2 = b_1^2 + \dots + b_q^2 = c \}\]
is homeomorphic to the product of spheres $S^{p-1} \times S^{q-1}$.  So, $A(p\mathbb{C}P^2 \sharp q\overline{\mathbb{C}P^2} ; \mathbb{R})$ is homeomorphic to $S^{p-1} \times S^{q-1} \times \mathbb{R}_{>0}$ by \eqref{eq:4.1-1} and hence to $S^{p-1} \times S^{q-1} \times \mathbb{R}$. 
 
(3) For each $i$, we change the variables in \eqref{eq:4.1-2} as follows:
 \begin{equation*}
  c_i = a_i + b_i,\ d_i = a_i - b_i.
 \end{equation*}
 Then one sees that $A(r(\mathbb{C}P^1 \times \mathbb{C}P^1);\mathbb{R})$ is homeomorphic to $A(r\mathbb{C}P^2 \sharp r\overline{\mathbb{C}P^2};\mathbb{R})$.  
\end{proof}

\begin{lem} \label{4.3}
For a finite set $A$, we denote the cardinality of $A$ by $|A|$.
Then
\begin{enumerate}
\item[(1)] $|A(p\mathbb{C}P^2 \sharp p\overline{\mathbb{C}P^2} ; \mathbb{Z}/2)|=2^{2p-1}-1$,
\item[(2)] $|A(r(\mathbb{C}P^1 \times \mathbb{C}P^1) ; \mathbb{Z}/2)|=2^{2r-1} + 2^{r-1} - 1$.
\end{enumerate}
\end{lem}

\begin{proof}
(1) By \eqref{eq:4.1-1}, we count the number of elements $(a_1,\dots,a_p,b_1,\dots,b_p)\in (\mathbb{Z}/2)^{2p}\backslash\{0\}$ satisfying $$a_1^2 + \dots + a_p^2 = b_1^2 + \dots + b_p^2.$$ This equation is equivalent to the existence of even number ofg$1$hin $a_1,\dots,a_p,b_1,\dots,b_p$.
Therefore,
 \[
|A(p\mathbb{C}P^2 \sharp p\overline{\mathbb{C}P^2} ; \mathbb{Z}/2)| + 1
   = \begin{pmatrix} 2p \\ 0 \end{pmatrix} + \begin{pmatrix} 2p \\ 2 \end{pmatrix} + \dots + \begin{pmatrix} 2p \\ 2p \end{pmatrix} = 2^{2p-1}.
\]    

(2) By \eqref{eq:4.1-2}, it is enough to show the following:
\begin{equation} \label{eq:4.2}
|\{(c_1, \dots ,c_r,d_1, \dots ,d_r) \in (\mathbb{Z}/2)^{2r} \mid c_1d_1 + \dots + c_rd_r =0 \}| = 2^{2r-1} + 2^{r-1}. 
\end{equation} 
We show this by induction.
When $r=1$, we can check \eqref{eq:4.2} easily.         
Suppose that \eqref{eq:4.2} holds when $r=k$, and we consider the case $r=k+1$.
When $c_{k+1}d_{k+1}=0$ (i.e., $(c_{k+1},d_{k+1})$ is $(0,0), (1,0)$ or $(0,1)$), the number of elements $(c_1,\dots,c_k,d_1,\dots,d_k)$ in $(\mathbb{Z}/2)^{2k}$ satisfying $c_1d_1 + \dots + c_kd_k=0$ is $2^{2k-1} + 2^{k-1}$ by assumption of induction.
When $c_{k+1}d_{k+1}=1$ (i.e., $(c_{k+1},d_{k+1})=(1,1)$), the number of elements $(c_1,\dots,c_k,d_1,\dots,d_k)$ in $(\mathbb{Z}/2)^{2k}$ satisfying $c_1d_1 + \dots + c_kd_k=1$ is $2^{2k} - (2^{2k-1} + 2^{k-1})$.
So
\[
  \begin{split}
 & |\{(c_1, \dots ,c_{k+1},d_1, \dots ,d_{k+1}) \in (\mathbb{Z}/2)^{2(k+1)} \mid c_1d_1 + \dots + c_{k+1}d_{k+1} =0 \}|\\
 =& 3(2^{2k-1} + 2^{k-1}) + 2^{2k} - (2^{2k-1} + 2^{k-1}) = 2^{2k+1} + 2^k. 
\end{split}
\]
Therefore \eqref{eq:4.2} also holds when $r=k+1$.
\end{proof}

Note that the manifolds in Proposition~\ref{prop:4.1} except $\mathbb{C}P^1\times\mathbb{C}P^1$ do not decompose into the product of two manifolds of positive dimension. The following theorem generalizes Theorem~\ref{theo:3.1}. 

\begin{thm} \label{theo:4.1}
Let $M_i$ $(1\le i\le k)$ and $M_j'$ $(1\le j\le \ell)$ be $\mathbb{C}P^1$ or the manifolds in Proposition~\ref{prop:4.1} except $\mathbb{C}P^1\times \mathbb{C}P^1$.    
If $H^*(\prod_{i=1}^kM_i;\mathbb{Z})$ and $H^*(\prod_{j=1}^\ell M_j^\prime;\mathbb{Z})$ are isomorphic as graded rings, then $k=\ell$ and there exists an element $\sigma$ in the symmetric group $S_k$ on $k$ letters such that $M_i$ and $M_{\sigma(i)}^\prime$ are diffeomorphic for all $1\le i\le k$.
\end{thm} 

\begin{proof}
 Let $m$ (resp, $m_{p,q}$, $n_r$ or $n$) be the number of $M_i$'s diffeomorphic to $\mathbb{C}P^1$ (resp, $p\mathbb{C}P^2 \sharp q\overline{\mathbb{C}P^2}$ $(p\ge q\ge 0,\ p+q\geq 1)$, $r(\mathbb{C}P^1 \times \mathbb{C}P^1)$ $(r\ge 2)$ or $S^4$).
 Similarly, let $m'$ (resp, $m'_{p,q}$, $n'_r$ or $n'$) be the number of $M'_j$'s diffeomorphic to $\mathbb{C}P^1$ (resp, $p\mathbb{C}P^2 \sharp q\overline{\mathbb{C}P^2}$ $(p\ge q\ge 0,\ p+q\geq 1)$, $r(\mathbb{C}P^1 \times \mathbb{C}P^1)$ $(r\ge 2)$ or $S^4$).
Therefore,
 \begin{equation} \label{eq:4.3}
  \begin{split}
   M&:=\prod_{i=1}^kM_i=(\mathbb{C}P^1)^m  \times \prod_{p\geq q} (p\mathbb{C}P^2 \sharp q\overline{\mathbb{C}P^2})^{m_{p,q}}\times \prod_{r \geq 2} (r(\mathbb{C}P^1 \times \mathbb{C}P^1))^{n_r} \times (S^4)^n\\
   M^\prime &:=\prod_{j=1}^\ell M_j^\prime=(\mathbb{C}P^1)^{m'}  \times \prod_{p \geq q} (p\mathbb{C}P^2 \sharp q\overline{\mathbb{C}P^2})^{m'_{p,q}}\times \prod_{r \geq 2} (r(\mathbb{C}P^1 \times \mathbb{C}P^1))^{n'_r} \times (S^4)^{n'}
  \end{split}
 \end{equation}
 
 By assumption, $H^*(M;\mathbb{Z})$ and $H^*(M^\prime;\mathbb{Z})$ are isomorphic as graded rings, and an isomorphism $\varphi$ between them induces an isomorphism between $H^*(M;R)$ and $H^*(M^\prime;R)$ for any commutative ring $R$ and induces a bijection between $A(M;R)$ and $A(M^\prime;R)$. 
 When $R=\mathbb{R}$, the bijection is a homeomorphism.   Comparing the homeomorphism type and the number of connected components of $A(M;\mathbb{R})$ and $A(M^\prime;\mathbb{R})$ using Lemmas~\ref{3.2} and \ref{4.2},  we obtain 
\begin{equation} \label{eq:4.4} 
2m+4m_{1,1}=2m'+4m'_{1,1}, \quad m_{p,q}=m'_{p,q} \ (p>q\ge 1),\quad m_{p,p}+n_p=m'_{p,p}+n'_p \ (p\ge 2). 
\end{equation}

The linear subspace spanned by all one dimensional connected components in $A(M;\mathbb{R})$ (resp, $A(M^\prime;\mathbb{R})$) is  $H^2((\mathbb{C}P^1)^m\times (\mathbb{C}P^2 \sharp \overline{\mathbb{C}P^2})^{m_{1,1}};\mathbb{R})$ (resp, $H^2((\mathbb{C}P^1)^{m'}\times (\mathbb{C}P^2 \sharp \overline{\mathbb{C}P^2})^{m'_{1,1}};\mathbb{R})$).
Therefore, the isomorphism $\varphi$ induces an isomorphism between $H^2((\mathbb{C}P^1)^m\times (\mathbb{C}P^2 \sharp \overline{\mathbb{C}P^2})^{m_{1,1}};\mathbb{Z})$ and $H^2((\mathbb{C}P^1)^{m'}\times(\mathbb{C}P^2 \sharp \overline{\mathbb{C}P^2})^{m'_{1,1}};\mathbb{Z})$.
In particular, $\varphi$ induces an isomorphism between the cohomology rings with $\mathbb{Z}/2$ coefficients.
It follows from Lemma~\ref{3.2} that 
\[ m|A(\mathbb{C}P^1;\mathbb{Z}/2)|+m_{1,1}|A(\mathbb{C}P^2 \sharp \overline{\mathbb{C}P^2};\mathbb{Z}/2)|=
m'|A(\mathbb{C}P^1;\mathbb{Z}/2)|+m'_{1,1}|A(\mathbb{C}P^2 \sharp \overline{\mathbb{C}P^2};\mathbb{Z}/2)| \]
and hence we have $m+m_{1,1}=m'+m_{1,1}'$ by Lemma~\ref{3.3}.  This together with the first identity in \eqref{eq:4.4} implies that 
\begin{equation} \label{eq:4.5}
m=m',\quad m_{1,1}=m_{1,1}'.
\end{equation}

The linear subspace spanned by all connected components homeomorphic to $S^{p-1}\times S^{p-1}\times \mathbb{R}$ $(p\ge 2)$ in $A(M;\mathbb{R})$ (resp, $A(M^\prime;\mathbb{R})$) is $H^2((p\mathbb{C}P^2 \sharp p\overline{\mathbb{C}P^2})^{m_{p,p}}\times (p(\mathbb{C}P^1 \times \mathbb{C}P^1))^{n_p};\mathbb{R})$ (resp, $H^2((p\mathbb{C}P^2 \sharp p\overline{\mathbb{C}P^2})^{m'_{p,p}}\times (p(\mathbb{C}P^1 \times \mathbb{C}P^1))^{n'_p};\mathbb{R})$).
Therefore, it follows from Lemma~\ref{3.2} that 
\[
\begin{split}
&m_{p,p}|A(p\mathbb{C}P^2 \sharp p\overline{\mathbb{C}P^2};\mathbb{Z}/2)|+n_p|A(p(\mathbb{C}P^1 \times \mathbb{C}P^1);\mathbb{Z}/2)|\\
 =&m'_{p,p}|A(p\mathbb{C}P^2 \sharp p\overline{\mathbb{C}P^2};\mathbb{Z}/2)|+n'_p|A(p(\mathbb{C}P^1 \times \mathbb{C}P^1);\mathbb{Z}/2)|
 \end{split}
 \]
 and hence we have 
\begin{equation} \label{eq:4.6}
m_{p,p}(2^{2p-1}-1)+n_p(2^{2p-1}+2^{p-1}-1)=m'_{p,p}(2^{2p-1}-1)+n'_p(2^{2p-1}+2^{p-1}-1)
\end{equation}
by Lemma~\ref{4.3}.   
So by \eqref{eq:4.4}, \eqref{eq:4.5}, and \eqref{eq:4.6}, we have 
\begin{equation} \label{eq:4.7}
m=m',\quad m_{p,q}=m'_{p,q}\ (p\ge q\ge 1),\quad n_p=n'_p\ (p\ge 2).
\end{equation}

It remains to prove $n=n'$ and $m_{p,0}=m'_{p,0}$ $(p\ge 1)$.
Since $H^*(M;\mathbb{Z})$ and $H^*(M';\mathbb{Z})$ are isomorphic by assumption, the Poincar\'{e} polynomials of $M$ and $M'$ must coincide. 
So, the Poincar\'{e} polynomials of $(S^4)^n \times \prod_{p\ge 1} (p\mathbb{C}P^2)^{m_{p,0}}$ and $(S^4)^{n'} \times \prod_{p\ge 1} (p\mathbb{C}P^2)^{m'_{p,0}}$ must coincide by \eqref{eq:4.3} and \eqref{eq:4.7}.
It follows that 
\[
(1+x^2)^n \times \prod_{p\ge 1}(1+px+x^2)^{m_{p,0}}=(1+x^2)^{n'} \times \prod_{p\ge 1}(1+px+x^2)^{m'_{p,0}}
\]
where $x$ is a variable.  This implies that $n=n'$ and $m_{p,0}=m'_{p,0}$.
\end{proof}

Similarly to Corollary 3.4, the following corollary follows from Theorem 4.4.

\begin{cor}[cancellation] 
Let $M$, $M^\prime$ and $M^{\prime \prime}$ be products of copies of $\mathbb{C}P^1$ and manifolds in Proposition~\ref{prop:4.1}. 
If $M \times M^{\prime \prime}$ and $M^\prime \times M^{\prime \prime}$ are diffeomorphic, then so are $M$ and $M^\prime$. 
\end{cor}

A \emph{topological toric manifold} introduced by Ishida-Fukukawa-Masuda (\cite{is-fu-ma13}) is a compact smooth manifold of real dimension $2n$ with a smooth action of complex torus $(\mathbb{C}^*)^n$ that is locally equivariantly diffeomorphic to a smooth faithful representation space of $(\mathbb{C}^*)^n$.
A toric manifold regarded as a smooth manifold is a topological toric manifold.
A topological toric manifold of real dimension two is diffeomorphic to $\mathbb{C}P^1$ and the manifolds in Proposition~\ref{prop:4.1} except $S^4$ are topological toric manifolds.  
Therefore, it follows from Theorem 4.4 that Theorem~\ref{theo:3.1} holds for topological toric manifolds, so we may ask the cohomological rigidity problem and the unique decomposition problem for topological toric manifolds and no counterexample is known even to these extended problems.  

\medskip
\noindent{\bf Acknowledgement}.
I would like to thank Mikiya Masuda for many interesting and fruitful discussions on this subject. 
I am also thankful to Hiroaki Ishida for his comments on Section 4.

\end{document}